\documentclass[12pt,a4paper,reqno]{amsart}

\usepackage{tikz}
\usetikzlibrary{matrix,arrows,calc,snakes,patterns,decorations.markings}

\usepackage[mathscr]{eucal}
\usepackage{graphics,epic}
\usepackage{amsfonts}
\usepackage{amscd}
\usepackage{latexsym}
\usepackage{amsmath,amssymb, amsthm, stmaryrd, bm, bbm}
\usepackage[all,2cell]{xy}
\usepackage{mathrsfs}

\newtheorem{theorem}{Theorem}[section]
\newtheorem*{theorem*}{Theorem}

\newtheorem{lemma}[theorem]{Lemma}

\newtheorem{corollary}[theorem]{Corollary}

\newtheorem*{conjecture*}{Conjecture}

\theoremstyle{remark}
\newtheorem{remark}[theorem]{Remark}
\newtheorem{example}[theorem]{Example}
\theoremstyle{definition}

\newcommand{\ie}{{\em i.e.~}\ }
\newcommand{\confer}{{\em cf.~}\ }

\newcommand{\opname}[1]{\operatorname{\mathsf{#1}}}

\renewcommand{\mod}{\opname{mod}\nolimits}

\newcommand{\grmod}{\opname{grmod}\nolimits}
\newcommand{\proj}{\mathrm{proj}}
\newcommand{\inj}{\mathrm{inj}}

\newcommand{\Mod}{\opname{Mod}\nolimits}

\newcommand{\rep}{\opname{rep}\nolimits}
\newcommand{\Rep}{\opname{Rep}\nolimits}

\newcommand{\thick}{\opname{thick}\nolimits}

\newcommand{\per}{\opname{per}\nolimits}

\newcommand{\id}{\mathrm{id}}

\newcommand{\Hom}{\opname{Hom}}

\newcommand{\ten}{\otimes}
\newcommand{\lten}{\overset{\opname{L}}{\ten}}

\newcommand{\ca}{{\mathcal A}}
\newcommand{\cb}{{\mathcal B}}
\newcommand{\cc}{{\mathcal C}}
\newcommand{\cd}{{\mathcal D}}

\newcommand{\cf}{{\mathcal F}}

\newcommand{\ch}{{\mathcal H}}

\numberwithin{equation}{section}

\begin{document}

\title[]{The derived category of an algebra with radical square zero}

\author{Dong Yang}
\address{Dong Yang, Department of Mathematics, Nanjing University, Nanjing 210093, P. R. China}
\email{yangdong@nju.edu.cn}

\date{\today}

\begin{abstract}
Koszul duality and covering theory are combined to realise the bounded derived category $\cd$ of an algebra with radical square zero as a certain orbit category of the bounded derived category of finitely presented representations of an associated infinite quiver. As a consequence, the possible shapes of the connected components of the Auslander--Reiten quiver of $\cd$ are described.
\end{abstract}

\maketitle

\section{Introduction}

Let $Q$ be a finite connected quiver and $A$ be the algebra with radical square zero associated to $Q$. The bounded derived category $\cd^b(\mod A)$ of finite-dimensional right $A$-modules was studied in \cite{BautistaLiu06,BautistaLiu17} using quiver representations and in \cite{BekkertDrozd09} using free boxes. In this paper we provide a third approach which is similar to but more structural and less explicit than the approach in \cite{BautistaLiu06,BautistaLiu17}: by Koszul duality $\cd^b(\mod A)$ is triangle equivalent to the perfect derived category of a graded hereditary algebra, which, by covering theory, is equivalent to a certain orbit category of the derived category of representations over an infinite quiver. More precisely, we construct an infinite quiver $P$ with a quiver-automorphism $\sigma$ and establish a triangle equivalence 
\[\cd^b(\rep^+(P))/\Sigma(\sigma^*)^{-1}\longrightarrow\cd^b(\mod A),\]
where $\rep^+(P)$ is the category of finitely presented representations over $P$, $\Sigma$ is the suspension functor, $\sigma^*$ is the pull-back along $\sigma$ and the category on the left is the orbit category of $\cd^b(\rep^+(P))$ with respect to the automorphism $\Sigma(\sigma^*)^{-1}$. This reduces the study of $\cd^b(\mod A)$ to the study of the representation theory of $P$ (representation theory of infinite quivers is studied in \cite{BautistaLiuPaquette13}). In particular, the above equivalence induces a bijection between the set of isomorphism classes of finitely presented indecomposable representations over $P$ and the set of isomorphism classes of indecomposable objects of $\cd^b(\mod A)$.  The quiver $P$ is opposite to the quiver $Q^{\Box}$ in \cite[Section 2]{BekkertDrozd09} and is the quiver $Q^{\mathbb{Z}}$ in \cite[Section 7]{BautistaLiu14}. Let $r\in\mathbb{N}\cup\{0\}$ be the grading period of $Q$ (which measures how far $Q$ is from being gradable, see Section~\ref{ss:quiver} for the precise definition) and $\tilde{Q}$ be the minimal gradable covering of $Q$ in the sense of \cite{BautistaLiu06}. If $r=0$, then $\tilde{Q}=Q$ and $P$ is the disjoint union of $\mathbb{Z}$ copies of $Q$; if $r\geq 1$, then $P$ is  the disjoint union of $r$ copies of $\tilde{Q}$. As a consequence, when $r=0$, there is a triangle equivalence 
\[
\cd^b(\rep(Q))\longrightarrow \cd^b(\mod A);
\]
when $r\geq 1$, there is a quiver-automorphism $s$ of $\tilde{Q}$ and a triangle equivalence 
\[\cd^b(\rep^+(\tilde{Q}))/\Sigma^r(s^*)^{-1}\longrightarrow\cd^b(\mod A).\]
The latter equivalence clarifies further the relationship between $\cd^b(\rep^+(\tilde{Q}))$ and $\cd^b(\mod A)$, which was first studied in \cite{BautistaLiu06,BautistaLiu17}.  As a consequence, we can use the general results on representation theory of infinite quivers established in \cite{BautistaLiuPaquette13} to describe the possible shapes of connected components of the Auslander--Reiten quiver of $\cd^b(\mod A)$, parallel to \cite[Section 5]{BautistaLiu17}.
   
We prove the above results more generally for graded quivers $Q$ satisfying certain conditions (see Section~\ref{ss:D-as-quiver-repre}), including non-positively graded quivers, with $\cd^b(\mod A)$ replaced by the thick subcategory $\cd$ of the derived category of right dg $A$-modules generated by $A/\mathrm{rad}(A)$ (Corollaries~\ref{cor:D-as-orbit-quiver-description}, \ref{cor:D-as-orbit-quiver-description-r=0}, and~\ref{cor:D-as-orbit-quiver-description-3} and Theorem~\ref{thm:AR-quiver}).

\smallskip
Throughout this paper let $k$ be a field.

\medskip
\noindent\emph{Acknowledgement.} The author acknowledges support by the National Natural Science Foundation of China No. 11401297 and a project funded by the Priority Academic Program Development of Jiangsu Higher Education Institutions. He thanks a referee for very helpful remarks.

\section{Orbit categories}

In this section we discuss orbit categories. Results in Sections~\ref{ss:contracting-orbits-1} and~\ref{ss:contracting-orbits-2} will be used in Section~\ref{ss:D-as-quiver-repre}: there is a quiver $P$ which is the disjoint union of copies of another quiver $\tilde{Q}$; we will apply Lemmas~\ref{lem:contracting-orbits} and~\ref{lem:contracting-objects-2} to identify a certain orbit category associated to $P$ with a suitable orbit category associated to $\tilde{Q}$. Results in Section~\ref{ss:AR-triangle-in-orbit-category} will be used in Section~\ref{ss:AR-quiver} to describe the possible shapes of the connected components of the Auslander--Reiten quiver of the bounded derived category of an algebra with radical square zero.

\medskip

Let $\ca$ be an additive $k$-category and $\Phi$ be a $k$-linear auto-equivalence of $\ca$. Let $\ca/\Phi$ be the \emph{orbit category} of $\ca$ with respect to $\Phi$ (\cite{Keller05}): it has the same objects as $\ca$, and the
morphism space from $X$ to $Y$ is defined by
\[\Hom_{\ca/\Phi}(X,Y)
=\bigoplus_{p\in\mathbb{Z}}\Hom_{\ca}(X,\Phi^p Y).\] 
If $\ca$ is idempotent complete, so is $\ca/\Phi$. In this case, an object $X$ is indecomposable in $\ca$ if and only if it is indecomposable in $\ca/\Phi$.

Let $\pi=\pi_\ca:\ca\to\ca/\Phi$ denote the projection functor.

\subsection{}\label{ss:contracting-orbits-1}
 Fix $r\in\mathbb{N}$ and let $\cb$ be the direct sum of $r$ copies of $\ca$. Precisely, the objects of $\cb$ are $r$-tuples $(X_0,\ldots,X_{r-1})$ of objects of $\ca$, and for two objects $(X_0,\ldots,X_{r-1})$ and $(Y_0,\ldots,Y_{r-1})$ the morphism space is 
\[
\Hom_{\cb}((X_0,\ldots,X_{r-1}),(Y_0,\ldots,Y_{r-1}))=\bigoplus_{i=0}^{r-1}\Hom_\ca(X_i,Y_i).
\]
For $0\leq i\leq r-1$, let $\ca_i$ be the full subcategory of $\cb$ consisting of  the objects $(X_0,\ldots,X_{r-1})$ with $X_j=0$ for $j\neq i$. Clearly $\ca_i$ is isomorphic to $\ca$.
Let $\Psi$ be a $k$-linear auto-equivalence of $\cb$ which restricts to equivalences $\Psi_i:\ca_i\to\ca_{i+1}$ for $0\leq i\leq r-1$ (where the indices are read modulo $r$). Then $\Psi^r$ restricts to a $k$-linear auto-equivalence of $\ca_0$, which induces a $k$-linear auto-equivalence of $\ca$. Let $\Phi$ be this auto-equivalence (in some sense $\Psi$ is an $r$-th root of $\Phi$). 
Let $\iota:\ca\to\ca_0\to\cb$ be the embedding of $\ca$ into $\cb$ as the $0$-th component. Then $\Psi^r \iota=\iota\Phi$.

\begin{lemma}\label{lem:contracting-orbits}
There is a commutative diagram of $k$-linear functors
\[
\xymatrix{\ca \ar[r]^\iota \ar[d]^{\pi_\ca} & \cb\ar[d]^{\pi_\cb}\\
\ca/\Phi \ar[r]^{\bar{\iota}} & \cb/\Psi,
}
\]
where $\bar{\iota}$ is an equivalence.
\end{lemma}
\begin{proof}
We define $\bar{\iota}:\ca/\Phi\to\cb/\Psi$: on objects $\bar{\iota}(X)=\iota(X)$ and on morphisms
\begin{align*}
\Hom_{\ca/\Phi}(X,Y)&=\bigoplus_{p\in\mathbb{Z}}\Hom_\ca(X,\Phi^p Y)\\
&\stackrel{\iota}{\to} \bigoplus_{p\in\mathbb{Z}}\Hom_{\cb}(\iota(X),\iota\Phi^p(Y))\\
&=\bigoplus_{p\in\mathbb{Z}}\Hom_{\cb}(\iota(X),\Psi^{pr}\iota(Y))\\
&\hookrightarrow\bigoplus_{p\in\mathbb{Z}}\Hom_{\cb}(\iota(X),\Psi^p\iota(Y))\\
&=\Hom_{\cb/\Psi}(\bar{\iota}(X),\bar{\iota}(Y)).
\end{align*}
It is fully faithful because $\iota$ is fully faithful and $\Hom_\cb(\iota(X),\Psi^p \iota(Y))=0$ unless $p$ is a multiple of $r$.
Moreover, for an object $(X_0,\ldots,X_{r-1})$ of $\cb$ there exist $X_1',\ldots,X_{r-1}'\in\ca$ such that $(X_0,\ldots,X_{r-1})$ is isomorphic in $\cb$ to $\iota(X_0)\oplus \Psi\iota(X_1')\oplus\ldots\oplus\Psi^{r-1} \iota(X'_{r-1})$, which is isomorphic to $\iota(X_0\oplus X'_1\oplus\ldots\oplus X'_{r-1})$ in $\cb/\Psi$. So $\bar{\iota}$ is dense. Therefore, $\bar{\iota}$ is a $k$-linear equivalence. The commutativity of the diagram is clear from the definition of $\bar{\iota}$.
\end{proof}

\subsection{}\label{ss:contracting-orbits-2}
Let $\cb$ be the direct sum of $\mathbb{Z}$ copies of $\ca$. Precisely, the objects of $\cb$ are $\mathbb{Z}$-tuples $(X_i)_{i\in\mathbb{Z}}$ of objects of $\ca$, where all but finitely many $X_i$ are $0$, and for two objects $X=(X_i)_{i\in\mathbb{Z}}$ and $Y=(Y_i)_{i\in\mathbb{Z}}$ the morphism space is
\[
\Hom_\cb(X,Y)=\bigoplus_{i\in\mathbb{Z}}\Hom_\ca(X_i,Y_i).
\]
For $i\in\mathbb{Z}$, let $\ca_i$ be the full subcategory of $\cb$ consisting of the objects $X=(X_j)_{j\in\mathbb{Z}}$ with $X_j=0$ for $j\neq i$. Then $\ca_i$ is isomorphic to $\ca$. Let $\Psi$ be a $k$-linear auto-equivalence of $\cb$ which restricts to equivalences $\Psi_i:\ca_i\to\ca_{i+1}$ for all $i\in\mathbb{Z}$. Let $\iota:\ca\to\ca_0\to\cb$ be the embedding of $\ca$ into $\cb$ as the $0$-th component.

\begin{lemma}\label{lem:contracting-objects-2}
The composition
\[
\bar{\iota}:\ca\stackrel{\iota}{\longrightarrow}\cb\stackrel{\pi_\cb}{\longrightarrow}\cb/\Psi
\]
is an equivalence.
\end{lemma}
\begin{proof}
On morphisms $\bar{\iota}$ is the embedding
\[
\Hom_\ca(X,Y)\longrightarrow\Hom_\cb(\iota(X),\iota(Y))\longrightarrow\bigoplus_{p\in\mathbb{Z}}\Hom_{\cb}(\iota(X),\Psi^p(\iota(Y))).
\]
This is an isomorphism because $\iota$ is fully faithful and $\Hom_\cb(\iota(X),\Psi^p(\iota(Y)))=0$ unless $p=0$. So $\bar{\iota}$ is fully faithful. Moreover, for an object $(X_i)_{i\in\mathbb{Z}}$ of $\ca$, there exist $X'_i\in \ca$ ($i\in\mathbb{Z}$) such that $(X_i)_{i\in\mathbb{Z}}$ is isomorphic in $\cb$ to $\bigoplus_{i\in\mathbb{Z}}\Psi^i\iota(X'_i)$ (note that this is a finite direct sum), which is isomorphic to $\iota(\bigoplus_{i\in\mathbb{Z}}X'_i)$ in $\cb/\Psi$. So $\bar{\iota}$ is dense. Therefore, $\bar{\iota}$ is an equivalence.
\end{proof}

\subsection{}\label{ss:AR-triangle-in-orbit-category}

Assume that $\ca$ is a Hom-finite Krull--Schmidt triangulated $k$-category and $\Phi$ is a $k$-linear triangle auto-equivalence of $\ca$. 
Assume further that
\begin{itemize}
\item[--] $\ca/\Phi$ is triangulated such that $\pi:\ca\to\ca/\Phi$ is a triangle functor;
\item[--] for any objects $X$ and $Y$ of $\ca$, the space $\Hom_\ca(X,\Phi^p Y)$ vanishes for almost all $p\in\mathbb{Z}$.
\end{itemize}
We remark that the first assumption is very strong. It is satisfied in the hereditary situation, see \cite[Theorems 1 and 6]{Keller05}. In general $\ca/\Phi$ is not triangulated. However, if $\ca$ is algebraic (\ie $\ca$ is triangle equivalent to the homotopy category of a pretriangulated dg category) and $\Phi$ admits a dg lift, then there is a triangulated $k$-category $\widetilde{\ca/\Phi}$, called a \emph{triangulated hull} of $\ca/\Phi$, such that $\ca/\Phi$ embeds into and generates $\widetilde{\ca/\Phi}$ and $\pi$ extends to a triangle functor $\ca\to\widetilde{\ca/\Phi}$, see \cite[Sections 5 and 9]{Keller05}.

It follows from the assumptions that $\ca/\Phi$ is Hom-finite and Krull--Schmidt, and there is a one-to-one correspondence between $\Phi$-orbits of isomorphism classes of indecomposable objects of $\ca$ and isomorphism classes of indecomposable objects of $\ca/\Phi$.
Moreover, the image of an Auslander--Reiten triangle in $\ca$ under $\pi$ is an Auslander--Reiten triangle in $\ca/\Phi$, and up to isomorphism all Auslander--Reiten triangles in $\ca/\Phi$ are of this form. Therefore, the Auslander--Reiten quiver of $\ca/\Phi$ is the orbit quiver of the Auslander--Reiten quiver of $\ca$ under the automorphism induced by $\Phi$.
We remark that here we do not require that Auslander--Reiten triangles exist for the whole category $\ca$.

\section{Quivers and derived categories}

In this section we recall the basic definitions and notation on graded quivers and derived category of dg modules. 

\subsection{Quivers}\label{ss:quiver}
Let $Q$ be a graded quiver (we consider an (ungraded) quiver as a graded quiver whose arrows are in degree $0$). We denote by $Q_0$ the set of vertices of $Q$ and $Q_1$ the set of arrows of $Q$. For $\alpha\in Q_1$, denote by $s(\alpha)$ and $t(\alpha)$ the source and the target of $\alpha$, respectively, and denote by $|\alpha|$ the degree of $\alpha$. For each $i\in Q_0$, we have a trivial path $e_i$ of degree $0$ with $s(e_i)=t(e_i)=i$.  
For $\alpha\in Q_1$, we introduce its formal inverse $\alpha^{-1}$ with $s(\alpha^{-1})=t(\alpha)$,  $t(\alpha^{-1})=s(\alpha)$ and $|\alpha^{-1}|=-|\alpha|$. 
A \emph{walk} is a formal product $w=c_m\cdots c_1$, where each $c_l$ is either a trivial path, or an arrow, or the inverse of an arrow such that $t(c_l)=s(c_{l+1})$ for $1\leq l\leq m-1$. Define $s(w)=s(c_1)$ as the source of $w$ and $t(w)=t(c_m)$ as the target of $w$. We say that $w$ is a \emph{closed walk} if $s(w)=t(w)$. The \emph{virtual degree} of $w$ is defined as $d(w)=\sum_{l=1}^m d(c_l)$, where for $\alpha\in Q_1$ we define $d(\alpha)=1-|\alpha|$ and $d(\alpha^{-1})=|\alpha|-1$ and for $i\in Q_0$ we define $d(e_i)=0$. The \emph{grading period} of $Q$ is defined as $0$ if all closed walks are of virtual degree $0$, and otherwise as the minimal positive integer $r$ such that there are closed walks of virtual degree $r$.  For example, the grading period of the Jordan quiver with the unique arrow in degree $-1$ is  $2$. Note that the virtual degree of any closed walk is a multiple of $r$. The definition of grading period is generalised from \cite[Definition 1.6]{BautistaLiu06} and \cite[Definition 7.3]{BautistaLiu14}, which was for $Q$ concentrated in degree $0$.

Assume that $Q$ is finite, \ie both $Q_0$ and $Q_1$ are finite sets. A non-trivial path of $Q$ is a walk $\alpha_m\cdots\alpha_1$, where all $\alpha_l$ are arrows. It is of degree $|\alpha_1|+\ldots+|\alpha_m|$. 
The \emph{path algebra} $kQ$ of $Q$ is the graded $k$-algebra which has all paths of $Q$ (including the trivial paths) as basis and the multiplication is given by concatenation of paths. So the degree $p$ component of $kQ$ has all paths of degree $p$ as basis. The \emph{complete path algebra} $\widehat{kQ}$ of $Q$ is the completion of $kQ$ in the category of graded $k$-algebras with respect to the $J$-adic topology, where $J$ is the ideal of $kQ$ generated by all arrows. So the degree $p$ component of $\widehat{kQ}$ consists of (possibly infinite) sums of all paths of degree $p$. For example, the complete path algebra of the Jordan quiver with the unique arrow in degree $2$ is isomorphic to the graded polynomial algebra $k[x]$ with $x$ in degree $2$. We refer to \cite[Section II.3]{CaenepeelVanOystaeyen88} for the theory on completions of graded rings.

\subsection{Derived categories}\label{ss:derived-category}

Let $A$ be a graded $k$-algebra, which is viewed as a dg $k$-algebra with trivial differential. Denote by $\cd(A)$ the derived category of right dg $A$-modules (for example, graded right $A$-modules are right dg $A$-modules). It is a triangulated $k$-category with suspension functor $\Sigma$ being the shift of complexes. See~\cite{Keller94,Keller06d}. For an object $M$ of $\cd(A)$, denote by $\thick(M)$ the thick subcategory of $\cd(A)$ generated by $M$, \ie the smallest triangulated subcategory of $\cd(A)$ which contains $M$ and which is closed under taking direct summands. Let $\per(A):=\thick(A_A)$ denote the thick subcategory of $\cd(A)$ generated by the free right dg $A$-module of rank $1$ and let $\cd_{fd}(A)$ denote the full (triangulated) subcategory of $\cd(A)$ consisting of those right dg $A$-modules with finite-dimensional total cohomology. If $A$ is concentrated in degree $0$ and is finite-dimensional over $k$, then $\cd(A)=\cd(\Mod A)$ and  there is a natural triangle equivalence between  $\cd^b(\mod A)$ and $\cd_{fd}(A)$.

\section{The derived category of an algebra with radical square zero}\label{s:radical-square-zero-alg}

In this section we combine Koszul duality and covering theory to construct two quivers $P$ and $\tilde{Q}$ and describe the derived category $\cd$ of an algebra $A$ with radical square zero as certain orbit categories of derived categories of representations over $P$ and $\tilde{Q}$. Moreover, we use the work \cite{BautistaLiuPaquette13} to describe the possible shapes of the connected components of the Auslander--Reiten quiver of $\cd$.

\medskip
Fix a finite graded quiver $Q$. Let $J$ be the ideal of the path algebra $kQ$ generated by all arrows and let $A=kQ/J^2$.  We denote by $\bar{J}$ the image of $J$ in $A$ under the surjective homomorphism $kQ\to A$. Let $K=kQ_0$, which is the direct product of finitely many copies of $k$. It is a subalgebra of $A$ and also there is a surjective homomorphism $A\to A/\bar{J}=K$. We view $K$ as a graded right $A$-module via this homomorphism and denote it by $S$ to avoid confusion.

Recall from Section~\ref{ss:derived-category} that $\cd(A)$ denotes the derived category of right dg $A$-modules. We emphasise that here we consider $A$ as a dg algebra with trivial differential. We are interested in the thick subcategory $\cd=\thick(S)$ of $\cd(A)$ generated by $S$, in particular, a description of the category in terms of quiver representations. 
If all arrows of $Q$ are in positive degrees (respectively, non-positive degrees), it follows from \cite[Lemma 4.3]{SuYang16} (respectively, \cite[Proposition 2.1 (b)]{KalckYang16}) that $\cd=\cd_{fd}(A)$. Recall  that if $A$ is concentrated in degree $0$, then $\cd_{fd}(A)$ is naturally triangle equivalent to $\cd^b(\mod A)$. This  case was studied in~\cite{BautistaLiu06,BekkertDrozd09}.

\subsection{Koszul duality}
In this subsection we recall Koszul duality \cite[Section 10]{Keller94} for $A$ as a dg algebra with trivial differential. Its connection to classical Koszul duality \cite{BeilinsonGinzburgSoergel96} on graded Koszul algebras can be found in \cite[Section 2.2]{Keller02} and \cite[Section 2.3]{LuPalmieriWuZhang08}.

Consider the tensor algebra
\[T_K(\bar{J}[1])=K\oplus \bar{J}[1]\oplus \bar{J}[1]\ten_K \bar{J}[1]\oplus\ldots=\bigoplus_{p\geq 0} (\bar{J}[1])^{\ten_K p},\]
where $J$ is viewed as complex with trivial differential and $[1]$ is the shift of complexes. 
The \emph{Koszul dual} $A^*$ of $A$ is defined as the graded $K$-dual of $T_K(\bar{J}[1])$, that is,
\begin{eqnarray}
A^*:=\widehat{T}_K(\Hom_K(\bar{J}[1],K)),\label{construction:koszul-dual}
\end{eqnarray}
where $\widehat{T}$ denotes the complete tensor algebra. 
This is a dg $k$-algebra whose multiplication is given by the tensor product and whose differential is induced by the multiplication of $\bar{J}$. Since the multiplication of $\bar{J}$ is trivial, it follows that the differential of $A^*$ is trivial. For example, if $Q$ is the Jordan quiver with the unique arrow in degree $-1$, then $A^*$ is isomorphic to the graded polynomial algebra $k[t]$ with $t$ in degree $2$.

By \cite[Lemma 3.7]{SuHao16}, there is a homotopically projective resolution $\mathbf{p}S$ of $S$ such that there is a quasi-isomorphism from $A^*$ to the dg endomorphism algebra of $\mathbf{p}S$. Thus it follows from \cite[Lemma 10.5, the ``exterior'' case]{Keller94} that the  triangle functor
\[
?\lten_{A^*}\mathbf{p}S:\cd(A^*)\longrightarrow\cd(A)
\]
restricts to a triangle equivalence
\[
\per(A^*)\longrightarrow\cd.
\]

\subsection{The description of $\cd$ in terms of graded modules}
\label{ss:derived-cat-as-orbit-cat}

We define a graded quiver $Q^{gr.op}$. Its underlying quiver is the opposite quiver of $Q$ and the grading is defined by $|\alpha^{op}|=1-|\alpha|$, where $\alpha^{op}$ is the arrow opposite to $\alpha$. Then the Koszul dual $A^*$ is exactly the complete path algebra $\widehat{kQ}{}^{gr.op}$ of $Q^{gr.op}$ viewed as a dg algebra with trivial differential. It follows that there is a triangle equivalence
\[
\per(\widehat{kQ}{}^{gr.op})\longrightarrow\cd.
\]

Let $\grmod(\widehat{kQ}{}^{gr.op})$ denote the category of finitely presented graded right $\widehat{kQ}{}^{gr.op}$-modules. This is an abelian category. Let $\langle 1\rangle$ denote the degree shifting. It extends to a triangle automorphism of $\cd^b(\grmod \widehat{kQ}{}^{gr.op})$, still denoted by $\langle 1\rangle$.
Applying \cite[Theorem 5.1 (e)]{KalckYang16a} to the graded hereditary algebra $\widehat{kQ}{}^{gr.op}$, we obtain that the orbit category $\cd^b(\grmod \widehat{kQ}{}^{gr.op})/\Sigma\langle-1\rangle$ is naturally triangulated such that the projection functor $\cd^b(\grmod\widehat{kQ}{}^{gr.op})\to\cd^b(\grmod \widehat{kQ}{}^{gr.op})/\Sigma\langle-1\rangle$ is a triangle functor and there is a triangle equivalence
\[
\cd^b(\grmod \widehat{kQ}{}^{gr.op})/\Sigma\langle-1\rangle\longrightarrow \per(\widehat{kQ}{}^{gr.op}).
\]
Therefore we have

\begin{theorem}\label{thm:derived-cat-as-orbit}
There is a triangle equivalence
\[
\cd^b(\grmod(\widehat{kQ}{}^{gr.op}))/\Sigma\langle-1\rangle\longrightarrow\cd.
\]
\end{theorem}

\begin{corollary}\label{c:indecomposable-object}
The composition
\[\grmod(\widehat{kQ}{}^{gr.op})\rightarrow\cd^b(\grmod(\widehat{kQ}{}^{gr.op}))\rightarrow\cd^b(\grmod(\widehat{kQ}{}^{gr.op}))/\Sigma\langle-1\rangle\to \cd\]
detects indecomposability (that is, an object of $\grmod(\widehat{kQ}{}^{gr.op})$ is indecomposable if and only if its image in $\cd$ is indecomposable) and induces a bijection on the isomorphism classes of objects.
\end{corollary}

\begin{proof} 
It suffices to show that the composition of the first two functors has the desired property. First, both the first two functors detect indecomposability. Secondly, let $M$ be an object of $\cd^b(\grmod(\widehat{kQ}{}^{gr.op}))$. Then there exist $M_1,\ldots,M_n\in\grmod(\widehat{kQ}{}^{gr.op})$ and $a_1,\ldots,a_n\in \mathbb{Z}$ such that $M$ is isomorphic to $\Sigma^{a_1}M_1\oplus\ldots\oplus \Sigma^{a_n}M_n$ in $\cd^b(\grmod (\widehat{kQ}{}^{gr.op}))$. So $M$ is isomorphic to $M_1\langle a_1\rangle\oplus\ldots\oplus M_n\langle a_n\rangle$ in $\cd^b(\grmod(\widehat{kQ}{}^{gr.op}))/\Sigma\langle-1\rangle$. Thirdly, let $M,~N\in\grmod(\widehat{kQ}{}^{gr.op})$ be isomorphic in $\cd^b(\grmod(\widehat{kQ}{}^{gr.op}))/\Sigma\langle-1\rangle$. Then there exist $f=(f_p)_{p\in\mathbb{Z}}$ and $g=(g_p)_{p\in\mathbb{Z}}$ with $f_p\in\Hom_{\cd^b(\grmod(\widehat{kQ}{}^{gr.op}))}(M,\Sigma^pN\langle-p\rangle)$ and $g_p\in\Hom_{\cd^b(\grmod(\widehat{kQ}{}^{gr.op}))}(N,\Sigma^pM\langle-p\rangle)$ such that $fg=\id_N$ and $gf=\id_M$. Since $\widehat{kQ}{}^{gr.op}$ is graded hereditary, we have $f_p=0$ and $g_p=0$ unless $p=0,1$. It follows that $f_0g_0=\id_N$ and $g_0f_0=\id_M$, namely, $M$ and $N$ are isomorphic in $\grmod(\widehat{kQ}{}^{gr.op})$. This completes the proof.
\end{proof}

In general the category $\cd$ is not Hom-finite.

\begin{lemma}\label{lem:path-algebra=complete-path-algebra} The following conditions are equivalent for $Q$:
\begin{itemize}
\item[(i)] $\widehat{kQ}{}^{gr.op}=kQ^{gr.op}$;
\item[(ii)] $\forall p\in\mathbb{Z}$, $Q^{gr.op}$ has only finitely many paths of degree $p$, equivalently, $Q$ has only finitely many paths of virtual degree $p$;
\item[(iii)] $\forall p\in\mathbb{Z}$, the degree $p$ component of $\widehat{kQ}{}^{gr.op}$ is finite-dimensional over $k$;
\item[(iv)] $\cd$ is Hom-finite and Krull--Schmidt.
\end{itemize}
\end{lemma}
\begin{proof}
The equivalences (i)$\Leftrightarrow$(ii)$\Leftrightarrow$(iii) follow directly from the definitions of $kQ^{gr.op}$ and $\widehat{kQ}{}^{gr.op}$.

(iii)$\Leftrightarrow$(iv): Because $\cd$ is triangle equivalent to $\per(\widehat{kQ}{}^{gr.op})$ and is idempotent complete, (iv) is equivalent to
\begin{itemize}
\item[(iv')] $\per(\widehat{kQ}{}^{gr.op})$ is Hom-finite.
\end{itemize}
Since $\per(\widehat{kQ}{}^{gr.op})$ is generated by $\widehat{kQ}{}^{gr.op}$, it follows by d\'evissage that (iv') is equivalent to
\begin{itemize}
\item[(iv'')] $\Hom_{\cd(\widehat{kQ}{}^{gr.op})}(\widehat{kQ}{}^{gr.op},\Sigma^p \widehat{kQ}{}^{gr.op})$ is finite-dimensional over $k$ for any $p\in\mathbb{Z}$.
\end{itemize}
This is equivalent to (iii) because 
\[
\Hom_{\cd(\widehat{kQ}{}^{gr.op})}(\widehat{kQ}{}^{gr.op},\Sigma^p \widehat{kQ}{}^{gr.op})=H^p(\widehat{kQ}{}^{gr.op})
\]
is the degree $p$ component of $\widehat{kQ}{}^{gr.op}$.
\end{proof}

\subsection{The description of $\cd$ in terms of quiver representations}\label{ss:D-as-quiver-repre}
Throughout this subsection, we assume the  equivalent definitions (i), (ii), (iii) and (iv) in Lemma~\ref{lem:path-algebra=complete-path-algebra}. Condition (iii) implies that $Q$ has no oriented cycles of virtual degree $0$, \ie oriented cycles whose length and degree are equal. 
These four conditions are satisfied in the following two cases:
\begin{itemize}
\item[--] all arrows of $Q$ are in non-positive degrees;
\item[--] all arrows of $Q$ are in positive degrees and there are no oriented cycles of virtual degree $0$, for example, if all arrows of $Q$ are in degrees $\geq 2$.
\end{itemize}
Recall that in these two cases $\cd=\cd_{fd}(A)$. In this subsection we improve and generalise some results of \cite{BautistaLiu06,BekkertDrozd09} on the study of $\cd$, which focus on the case when $Q$ is concentrated in degree $0$. More precisely, we will construct two (ungraded) quivers $P$ and $\tilde{Q}$ and realise $\cd$ as a certain orbit category of the bounded derived category of finitely presented representations of $P$ (respectively, $\tilde{Q}$). We point out that the constructions of $P$ and $\tilde{Q}$ and  Lemmas~\ref{lem:walks-in-Q-vs-walks-in-P},~\ref{lem:from-modules-to-representations} and \ref{lem:relation-between-P-and-Qtilde} work for arbitrary $Q$; however, in general $P$ and $\tilde{Q}$ are not strongly locally finite.

\medskip

We first define the quiver $P$, which is infinite. Its vertices are pairs $(i,j)$, where $i\in Q_0$ and $j\in\mathbb{Z}$. Its arrows are of the form $(\alpha,j):(s(\alpha),j)\to (t(\alpha),j+d(\alpha))$, where $\alpha\in Q_1$. 

\begin{lemma}\label{lem:walks-in-Q-vs-walks-in-P}
Let $i,i'\in Q_0$ and $j,j'\in\mathbb{Z}$. The map $(\alpha,j)\mapsto \alpha$ defines a bijection from the set of arrows in $P$ from $(i,j)$ to $(i',j')$ to the set of arrows in $Q$ from $i$ to $i'$ of virtual degree $j'-j$. It extends to a bijection from the set of paths (respectively, walks) in $P$ from $(i,j)$ to $(i',j')$ to the set of paths (respectively, walks) in $Q$ from $i$ to $i'$ of virtual degree $j'-j$.
\end{lemma}

As a consequence of Lemma~\ref{lem:walks-in-Q-vs-walks-in-P}, $P$ has no oriented cycles because $Q$ has no oriented cycles of virtual degree $0$; and $P$ is a strongly locally finite quiver in the sense of \cite{BautistaLiuPaquette13}. The assignments $(i,j)\mapsto (i,j+1)$ and $(\alpha,j)\mapsto (\alpha,j+1)$ define a quiver-automorphism $\sigma$ of  $P$. 

\begin{example}\label{ex:kronecker-1}
Let $Q$ be the graded Kronecker quiver
\[
\xymatrix@C=3pc{
1\ar@<.7ex>[r]^\alpha \ar@<-.7ex>[r]_\beta & 2
}
\]
with $|\alpha|=1$ and $|\beta|=-1$. Then $d(\alpha)=0$, $d(\beta)=2$ and $P$ is the quiver
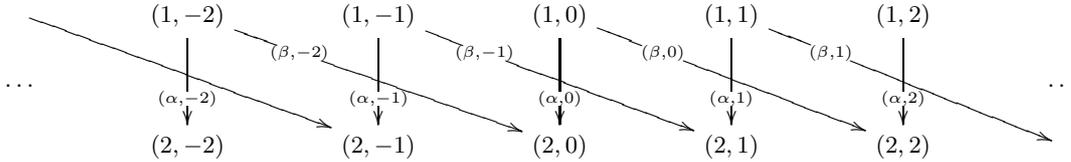
\[{\scriptsize
\xymatrix@R=1pc@C=3pc{
\ar[ddrr] & (1,-2)\ar[dd]|(0.65){(\alpha,-2)} \ar[ddrr]|(0.3){(\beta,-2)} & (1,-1)\ar[dd]|(0.65){(\alpha,-1)} \ar[ddrr]|(0.3){(\beta,-1)} & (1,0)\ar[dd]|(0.65){(\alpha,0)} \ar[ddrr]|(0.3){(\beta,0)} & (1,1)\ar[dd]|(0.65){(\alpha,1)} \ar[ddrr]|(0.3){(\beta,1)}  & (1,2)\ar[dd]|(0.65){(\alpha,2)}  & \\
\ldots & & & & & & \ldots\\
& (2,-2) & (2,-1) & (2,0) & (2,1) & (2,2) &  &&
}
}
\]
and $\sigma$ is shifting to the right by one step.
\end{example}

Let $\Rep(P)$ denote the category of representations of $P$ over $k$ and $\rep^+(P)$ (respectively, $\rep^-(P)$) denote the full subcategory of finitely presented representations (respectively, finitely co-presented representations), see \cite[Definition 1.5]{BautistaLiuPaquette13}. 
According to \cite[Proposition 1.15]{BautistaLiuPaquette13}, $\rep^+(P)$ (respectively, $\rep^-(P)$) is a Hom-finite and Ext-finite hereditary abelian $k$-category. 
Let $\sigma^*$ be the automorphism of $\Rep(P)$ defined by the pull-back along $\sigma$. It restricts to an automorphism of $\rep^+(P)$ (respectively, $\rep^-(P)$) and induces a triangle automorphism of $\cd^b(\rep^+(P))$ (respectively, $\cd^b(\rep^-(P))$), still denoted by $\sigma^*$. The following is standard.

\begin{lemma}\label{lem:from-modules-to-representations}
There is a $k$-linear isomorphism $\grmod kQ^{gr.op}\to \rep^+(P)$ and a $k$-linear triangle isomorphism $\cd^b(\grmod kQ^{gr.op})\to\cd^b(\rep^+(P))$ which take $\langle 1\rangle$ to $\sigma^*$.
\end{lemma}
\begin{proof}
We define the isomorphism on objects only. For $M\in\grmod kQ^{gr.op}$, the corresponding representation $V$ of $P$ is defined by
\begin{align*}
V_{(i,j)}&=(Me_i)^j~~\text{for}~i\in Q_0, j\in\mathbb{Z},\\
V_{(\alpha,j)}&: V_{(s(\alpha),j)}\to V_{(t(\alpha),j+1-|\alpha|)},~~m\mapsto m\alpha^{op},~~\text{for}~\alpha\in Q_1,j\in\mathbb{Z},
\end{align*}
where $\alpha^{op}$ is the arrow of $Q^{gr.op}$ opposite to $\alpha$.
\end{proof}

It follows from this lemma and the results in Section~\ref{ss:derived-cat-as-orbit-cat} that the orbit category $\cd^b(\rep^+(P))/\Sigma(\sigma^*)^{-1}$ is naturally triangulated such that the projection functor $\cd^b(\rep^+(P))\to\cd^b(\rep^+(P))/\Sigma(\sigma^*)^{-1}$ is a triangle functor. Moreover,

\begin{corollary}\label{cor:D-as-orbit-quiver-description}
\begin{itemize}
\item[(a)]
There is a triangle equivalence
\[\cd^b(\rep^+(P))/\Sigma(\sigma^*)^{-1}\longrightarrow\cd.\]
\item[(b)]
The composition
\[\rep^+(P)\rightarrow\cd^b(\rep^+(P))\rightarrow\cd^b(\rep^+(P))/\Sigma(\sigma^*)^{-1}\to \cd\]
detects indecomposability and induces a bijection on the isomorphism classes of objects. 
\end{itemize}
\end{corollary}

Let $\nu_P:\proj(P)\to\inj(P)$ be the Nakayama functor given in \cite[Proposition 1.19]{BautistaLiuPaquette13}, which is an equivalence. It extends to a triangle equivalence $\ch^b(\proj(P))\to\ch^b(\inj(P))$ of bounded homotopy categories. Since both $\rep^+(P)$ and $\rep^-(P)$ are hereditary, the natural embeddings $\ch^b(\proj(P))\to \cd^b(\rep^+(P))$ and $\ch^b(\inj(P))\to\cd^b(\rep^-(P))$ are both triangle equivalences. It follows that there is a triangle equivalence $\mathbb{\nu}:\cd^b(\rep^+(P))\to\cd^b(\rep^-(P))$, which commutes with the pull-back $\sigma^*$ because $\nu_P$ commutes with $\sigma^*$. So the orbit category $\cd^b(\rep^-(P))/\Sigma(\sigma^*)^{-1}$ is naturally triangulated such that the projection functor $\cd^b(\rep^-(P))\to\cd^b(\rep^-(P))/\Sigma(\sigma^*)^{-1}$ is a triangle functor. Moreover,

\begin{corollary}\label{cor:D-as-orbit-quiver-description-2}
\begin{itemize}
\item[(a)]
There is a triangle equivalence
\[\cd^b(\rep^-(P))/\Sigma(\sigma^*)^{-1}\longrightarrow\cd.\]
\item[(b)]
The composition
\[\rep^-(P)\rightarrow\cd^b(\rep^-(P))\rightarrow\cd^b(\rep^-(P))/\Sigma(\sigma^*)^{-1}\to \cd\]
detects indecomposability and induces a bijection on the isomorphism classes of objects. 
\end{itemize}
\end{corollary}

\medskip

Assume in the rest of this subsection that $Q$ is connected. We fix a vertex $i\in Q_0$ and define the quiver $\tilde{Q}$. The vertices of $\tilde{Q}$ are equivalence classes of walks of $Q$ whose source is $i$, where two walks are equivalent if their targets are the same and they have the same virtual degree. The equivalence class of a walk $u$ is denoted by $[u]$. The arrows of $\tilde{Q}$ are of the form $\alpha^{[u]}: [u]\to [\alpha u]$, where $[u]$ runs over all vertices of $\tilde{Q}$ and $\alpha$ runs over all arrows of $Q$ with $s(\alpha)=t(u)$. This is a connected quiver. It has no oriented cycles because $Q$ has no oriented cycles of virtual degree $0$. The construction of $\tilde{Q}$ is generalised from \cite[Theorem 1.3]{BautistaLiu06}, where $Q$ is concentrated in degree $0$ and $\tilde{Q}$ is the minimal gradable covering of $Q$. However, in our general setting $\tilde{Q}$ is in general not gradable in the sense of \cite[Definition 1.1]{BautistaLiu06} and \cite[Definition 7.1]{BautistaLiu14}.

\begin{example}\label{ex:kronecker-2}
Let $Q$ be the graded Kronecker quiver
\[
\xymatrix@C=3pc{
1\ar@<.7ex>[r]^\alpha \ar@<-.7ex>[r]_\beta & 2
}
\]
with $|\alpha|=1$ and $|\beta|=-1$. Then $\tilde{Q}$ is the quiver
\[{\scriptsize
\xymatrix@R=1pc@C=3pc{
\ar[ddr] & [\beta^{-1}\alpha]\ar[dd]|{\alpha^{[\beta^{-1}\alpha]}}\ar[ddr]|{\beta^{[\beta^{-1}\alpha]}} & [e_1]\ar[dd]|{\alpha^{[e_1]}}\ar[ddr]|{\beta^{[e_1]}} & [\alpha^{-1}\beta] \ar[dd]|{\alpha^{[\alpha^{-1}\beta]}}\ar[ddr]|{\beta^{[\alpha^{-1}\beta]}}& [\alpha^{-1}\beta\alpha^{-1}\beta]\ar[dd]|{\alpha^{[\alpha^{-1}\beta\alpha^{-1}\beta]}} \ar[ddr]& & \\
 \ldots&&&&&\ldots\\
 & [\alpha\beta^{-1}\alpha] & [\alpha] & [\beta] & [\beta\alpha^{-1}\beta] & 
}
}
\]
\end{example}

\bigskip

Let $r$ be the grading period of $Q$. For the graded quiver in Examples~\ref{ex:kronecker-1} and \ref{ex:kronecker-2}, the grading period is $2$ and $P$ is the disjoint union of $2$ copies of $\tilde{Q}$. This is a general phenomenon.

\begin{lemma}\label{lem:relation-between-P-and-Qtilde} If $r=0$, then $\tilde{Q}$ is isomorphic to $Q$ as an ungraded quiver and $P$ is isomorphic to the disjoint union of $\mathbb{Z}$ copies of $Q$. If $r\geq 1$, then $P$ is isomorphic to the disjoint union of $r$ copies of $\tilde{Q}$. 
\end{lemma}
\begin{proof} 

For $j\in\mathbb{Z}$ let $P^j$ be the connected component of $P$ containing the vertex $(i,j)$. Then the quiver-automorphism $\sigma$ of $P$ restricts to quiver-isomorphisms $P^j\to P^{j+1}$ for any $j\in\mathbb{Z}$. Moreover, since to every vertex of $Q$ there is a walk from $i$, it follows  by Lemma~\ref{lem:walks-in-Q-vs-walks-in-P} that to every vertex of $P$ there is a walk from $(i,j)$ for some $j\in\mathbb{Z}$. Therefore, $P=\cup_{j\in\mathbb{Z}} P^j$.

We show that $\tilde{Q}$ is isomorphic to $P^0$. Define a quiver-morphism $\rho:\tilde{Q}\to P$ by $\rho([u])=(t(u),d(u))$ and $\rho(\alpha^{[u]})=(\alpha,t(u))$. This map is injective on both vertices and arrows because an equivalence class of walks is determined by its target and its virtual degree. By Lemma~\ref{lem:walks-in-Q-vs-walks-in-P}, $(i',j)$ is a vertex of $P^0$ if and only if there is a walk $u$ in $Q$ from $i$ to $i'$ of virtual degree $j$. In this case, $(i',j)=\rho([u])$. Moreover, $(\alpha,j)$ is an arrow of $P^0$ if and only if $(s(\alpha),j)$ is a vertex of $P^0$ if and only if there is a walk $u$ in $Q$ from $i$ to $s(\alpha)$ of virtual degree $j$. In this case, $(\alpha,j)=\rho(\alpha^{[u]})$. Therefore, the image of $\rho$ is exactly $P^0$ and $\rho$ induces a quiver-isomorphism from $\tilde{Q}$ to $P^0$. 

(1) Assume that $r=0$. Then it follows from Lemma~\ref{lem:walks-in-Q-vs-walks-in-P} that for $j\neq 0$ the vertex $(i,j)$ does not belong to $P^0$, because there are no closed walks in $Q$ at $i$ of virtual degree $j$. So $P^j$ is different from $P^0$ for $j\neq 0$ and by applying the quiver-automorphism $\sigma$ we deduce that the $P^j$'s ($j\in\mathbb{Z}$) are pairwise different but all of them are isomorphic to $\tilde{Q}$. So  $P$ is isomorphic to the disjoint union of $\mathbb{Z}$ copies of $\tilde{Q}$.

Moreover, for any vertex $i'$ of $Q$, all walks from $i$ to $i'$ are of the same virtual degree. Therefore an equivalence class of walks is determined by its target. It follows that the quiver-morphism $\tilde{Q}\to Q$, $[u]\mapsto t(u),~\alpha^{[u]}\mapsto \alpha$ is an isomorphism.
As a consequence, $P$ is isomorphic to the disjoint union of $\mathbb{Z}$ copies of $Q$.

(2) Assume that $r\geq 1$. Then it follows from Lemma~\ref{lem:walks-in-Q-vs-walks-in-P} that for $j\in\mathbb{Z}$ the vertex $(i,j)$ belongs to $P^0$ if and only if $j$ is a multiple of $r$, because all closed walks in $Q$ at $i$ are of virtual degree a multiple of $r$. So $P^j=P^0$ if and only if $j$ is a multiple of $r$ and by applying the quiver-automorphism $\sigma$ we deduce that $P^j=P^{j'}$ if and only if $j-j'$ is a multiple of $r$. Hence $P$ is the disjoint union of $P^0,\ldots,P^{r-1}$, all of which are isomorphic to $\tilde{Q}$. Therefore $P$ is isomorphic to the disjoint union of $r$ copies of $\tilde{Q}$.
\end{proof}

\begin{remark}\label{rem:covering-is-infinite}\label{ex:covering-of-type-affine-A}
\begin{itemize}
\item[(a)]
By Lemma~\ref{lem:relation-between-P-and-Qtilde}, $\tilde{Q}$ is finite if $r=0$ and infinite if $r\geq 1$.
\item[(b)]
The quiver-morphism $\tilde{Q}\to Q$, $[u]\mapsto t(u),~\alpha^{[u]}\mapsto \alpha$ is a covering (see for example \cite[Section 1]{BautistaLiu06} for the definition of covering). 
\item[(c)] 
Assume $r\geq 1$. Then $Q$ is of type $\tilde{\mathbb{A}}$ if and only if $\tilde{Q}$ is of type $\mathbb{A}^\infty_\infty$; $Q$ is a graded oriented cycle if and only if $\tilde{Q}$ is the linear quiver of type $\mathbb{A}^\infty_\infty$.
\end{itemize}
\end{remark}

We identify $\tilde{Q}$ as the connected component $P^0$ of $P$ containing the vertex $(i,0)$ via the quiver-morphism $\rho$ defined in the proof of Lemma~\ref{lem:relation-between-P-and-Qtilde}.

Assume that $r=0$. We identify $Q$ with $\tilde{Q}$. Then $Q$ has no oriented cycles and $P$ is the disjoint union of $\sigma^j Q$, $j\in\mathbb{Z}$. So $\cd^b(\rep^+(P))$ is the direct sum of $\mathbb{Z}$ copies of $\cd^b(\rep (Q))$. Here we forget the grading on $Q$ and denote by $\rep(Q)$ the category of finite-dimensional representations over $Q$. It follows from Lemma~\ref{lem:contracting-objects-2} that the composition
\[
\cd^b(\rep(Q))\to \cd^b(\rep^+(P))\to\cd^b(\rep^+(P))/\Sigma(\sigma^*)^{-1},
\]
where the first functor is extension by $0$, is a triangle equivalence. So by Corollary~\ref{cor:D-as-orbit-quiver-description} we have the following corollary.

\begin{corollary}\label{cor:D-as-orbit-quiver-description-r=0} Assume that $r=0$. 
Then there is a triangle equivalence 
\[
\cd^b(\rep(Q))\longrightarrow \cd.
\]
\end{corollary}

Assume that $r\geq 1$. Then $P$ is the disjoint union of $\sigma^j \tilde{Q}$, $j=0,\ldots,r-1$. So $\cd^b(\rep^+(P))$ is the direct sum of $r$ copies of $\cd^b(\rep^+(\tilde{Q}))$. Let $w$ be a closed walk of $Q$ at the vertex $i$ of degree $r$. Then the map $[u]\to [uw]$ extends to a quiver isomorphism $s$ of $\tilde{Q}$, which is exactly $\sigma^r$ restricted to $\tilde{Q}$. It follows from Lemma~\ref{lem:contracting-orbits} that there is a commutative diagram of $k$-linear functors
\[
\xymatrix{
\cd^b(\rep^+(\tilde{Q}))\ar[r]\ar[d] & \cd^b(\rep^+(P))\ar[d]\\
\cd^b(\rep^+(\tilde{Q}))/\Sigma^r (s^*)^{-1}\ar[r] & \cd^b(\rep^+(P))/\Sigma(\sigma^*)^{-1},
}
\]
where the top horizontal functor is extension by $0$,  the vertical functors are the projection functors and the bottom functor is an equivalence.  Thus on $\cd^b(\rep^+(\tilde{Q}))/\Sigma^r (s^*)^{-1}$ there is a triangle structure such that the above diagram becomes a commutative diagram of $k$-linear triangle functors. Similarly, there is a commutative diagram of $k$-linear triangle functors
\[
\xymatrix{
\cd^b(\rep^-(\tilde{Q}))\ar[r]\ar[d] & \cd^b(\rep^-(P))\ar[d]\\
\cd^b(\rep^-(\tilde{Q}))/\Sigma^r (s^*)^{-1}\ar[r] & \cd^b(\rep^-(P))/\Sigma(\sigma^*)^{-1},
}
\]
where the top horizontal functor is extension by $0$, the vertical functors are the projection functors and the bottom functor is an equivalence. The next result follows from Corollaries~\ref{cor:D-as-orbit-quiver-description} and~\ref{cor:D-as-orbit-quiver-description-2}.

\begin{corollary} \label{cor:D-as-orbit-quiver-description-3}
Assume that $r\geq 1$. 
Then there are triangle equivalences
\[
\cd^b(\rep^+(\tilde{Q}))/\Sigma^r(s^*)^{-1}\longrightarrow\cd
\]
and
\[
\cd^b(\rep^-(\tilde{Q}))/\Sigma^r(s^*)^{-1}\longrightarrow\cd.
\]
\end{corollary}

\subsection{The Auslander--Reiten quiver of $\cd$}\label{ss:AR-quiver}

Throughout this subsection we assume that $Q$ is connected and the equivalent conditions (i), (ii), (iii) and (iv) in Lemma~\ref{lem:path-algebra=complete-path-algebra} hold. We  will use results of \cite{BautistaLiuPaquette13} on the Auslander--Reiten theory of infinite quivers to describe the possible shapes of the connected components of the Auslander--Reiten quiver of $\cd$.

\begin{lemma}\label{lem:existence-of-AR-triangle-and-oritented-cycles}
The following conditions are equivalent:
\begin{itemize}
\item[(1)] $P$ (respectively, $\tilde{Q}$) has no left infinite paths;
\item[(2)] $P$ (respectively, $\tilde{Q}$) has no right infinite paths;
\item[(3)] $Q$ has no oriented cycles;
\item[(4)] $A$ as an ordinary algebra has finite global dimension;
\item[(5)] $\cd$ has Auslander--Reiten triangles.
\end{itemize}
\end{lemma}
\begin{proof}
The equivalence between (3) and (4) is well-known. 

(3)$\Rightarrow$(1): Assume that $P$ has a left infinite path. Because $Q_0$ is finite, there is an $i\in Q_0$ and $j, j'\in\mathbb{Z}$ such that there is a non-trivial path in $P$ from $(i,j)$ to $(i,j')$. By Lemma~\ref{lem:walks-in-Q-vs-walks-in-P}, there is a non-trivial path in $Q$ from $i$ to $i$, namely, an oriented cycle of $Q$ at $i$.

(1)$\Rightarrow$(3): Assume that $Q$ has an oriented cycle at a vertex $i$. Then by Lemma~\ref{lem:walks-in-Q-vs-walks-in-P} there exist $j,j'\in\mathbb{Z}$ such that there is a non-trivial path $u$ from $(i,j)$ to $(i,j')$. So $u\sigma^{j-j'}(u)\sigma^{2j-2j'}(u)\cdots$ is a left infinite path of $P$.

(2)$\Leftrightarrow$(3): Similar to (1)$\Leftrightarrow$(3).

[(1)+(2)]$\Leftrightarrow$(5): 
By \cite[Theorem 7.11]{BautistaLiuPaquette13}, [(1)+(2)] is equivalent to that $\cd^b(\rep^+(\tilde{Q}))$ has Auslander--Reiten triangles. By Corollary~\ref{cor:D-as-orbit-quiver-description-3} and Section~\ref{ss:AR-triangle-in-orbit-category}, the latter condition is equivalent to (5).
\end{proof}

Let $r$ be the grading period of $Q$. If $r=0$, then $\cd$ is triangle equivalent to $\cd^b(\rep(Q))$ by Corollary~\ref{cor:D-as-orbit-quiver-description-r=0}. So the Auslander--Reiten theory of $\cd$ is the same as the Auslander--Reiten theory of $\cd^b(\rep(Q))$. 
In the rest of this subsection we assume that $r\geq 1$. Recall that $P$ and $\tilde{Q}$ are strongly locally finite infinite quivers with quiver-automorphisms $\sigma$ and $s$, respectively. Let $\Gamma_\cd$ (respectively, $\Gamma_{\rep^+(\tilde{Q})}$, $\Gamma_{\cd^b(\rep^+(\tilde{Q}))}$, $\Gamma_{\cd^b(\rep^+(P))}$) denote the Auslander--Reiten quiver of $\cd$ (respectively, $\rep^+(\tilde{Q})$, $\cd^b(\rep^+(\tilde{Q}))$, $\cd^b(\rep^+(P))$). As a consequence of Corollaries~\ref{cor:D-as-orbit-quiver-description} and~\ref{cor:D-as-orbit-quiver-description-3}, we have

\begin{corollary}\label{cor:AR-quiver}
$\Gamma_\cd$ is the orbit quiver of $\Gamma_{\cd^b(\rep^+(P))}$ by the quiver-automorphism induced by $\Sigma(\sigma^*)^{-1}$ and is the orbit quiver of $\Gamma_{\cd^b(\rep^+(\tilde{Q}))}$ by the quiver-automorphism induced by $\Sigma^r (s^*)^{-1}$.
\end{corollary}

The Auslander--Reiten theory of $\rep^+$ and $\cd^b(\rep^+)$ of a strongly locally finite quiver is studied in \cite{BautistaLiuPaquette13}. We refer to \cite[Section 4]{BautistaLiuPaquette13} for the definitions of right-most sections and of finite wings.

\begin{theorem}\label{thm:AR-quiver}
\begin{itemize}
\item[(a)] 
If $Q$ has no oriented cycles, then $\Gamma_\cd$ has $r$ components of shape $\mathbb{Z}\tilde{Q}^{op}$ and all other components are of shape $\mathbb{Z}\mathbb{A}_\infty$. When $Q$ is of type $\tilde{\mathbb{A}}$, the number of components of shape $\mathbb{Z}\mathbb{A}_\infty$ is $2r$; otherwise, this number is infinite.
\item[(b)]
If $Q$ is a graded oriented cycle, then $\Gamma_\cd$ consists of $r$ components of shape $\tilde{Q}^{op}$ and $r$ components of shape $\mathbb{Z}\mathbb{A}_\infty$.
\item[(c)] 
If $Q$ has oriented cycles but is not a graded oriented cycle, then $\Gamma_\cd$ has $r$ isomorphic components which is a full subquiver of $\mathbb{Z}\tilde{Q}^{op}$ with a right-most section and with finite Auslander--Reiten orbits only, and all other components are of shape $\mathbb{Z}\mathbb{A}_\infty$, $\mathbb{N}^-\mathbb{A}_\infty$, $\mathbb{N}\mathbb{A}_\infty$ or finite wings. There are infinitely many finite wings.
\end{itemize}
\end{theorem}

\begin{proof}
Let $\mathcal{R}$ be the disjoint union of the regular components of $\Gamma_{\rep^+(\tilde{Q})}$. Then according to \cite[Theorem 7.10]{BautistaLiuPaquette13}, $\Gamma_{\cd^b(\rep^+(\tilde{Q}))}$ is the disjoint union of $\Sigma^p\mathcal{R}$ and $\Sigma^p\cc_{\tilde{Q}}$, where $p$ runs over all integers and $\cc_{\tilde{Q}}$ is the connecting component, which is glued from the preprojective component and preinjective components of $\Gamma_{\rep^+(\tilde{Q})}$. Because $s^*$ restricts to automorphisms of $\mathcal{R}$ and of $\cc_{\tilde{Q}}$, it follows from Corollary~\ref{cor:AR-quiver} that $\Gamma_\cd$ is the disjoint union of $\Sigma^p\mathcal{R}$ and $\Sigma^p\cc_{\tilde{Q}}$, where $p=0,\ldots,r-1$.

If $Q$ is of type $\tilde{\mathbb{A}}$, then $\tilde{Q}$ is of type $\mathbb{A}^\infty_\infty$. If $Q$ is not of type $\tilde{\mathbb{A}}$, then $\tilde{Q}$ is not of type $\mathbb{A}^\infty_\infty$. Since $\tilde{Q}$ has a quiver-automorphism $s$ of infinite order while quivers of type $\mathbb{D}_\infty$ have no such quiver-automorphisms, it follows that $\tilde{Q}$ is not of infinite Dynkin type. 

(a) Assume that $Q$ has no oriented cycles. Then $\tilde{Q}$ has no infinite paths, by Lemma~\ref{lem:existence-of-AR-triangle-and-oritented-cycles}. So $\cc_{\tilde{Q}}$ is of shape $\mathbb{Z}\tilde{Q}^{op}$, by \cite[Proposition 7.9(1)]{BautistaLiuPaquette13}. By \cite[Corollary 4.16]{BautistaLiuPaquette13},  all components of $\mathcal{R}$ are of shape $\mathbb{Z}\mathbb{A}_\infty$. The first statement of (a) follows. 

If $Q$ is of type $\tilde{\mathbb{A}}$, then $\tilde{Q}$ is of type $\mathbb{A}^\infty_\infty$. According to \cite[Theorem 5.17(2)]{BautistaLiuPaquette13}, $\mathcal{R}$ consists of two components of shape $\mathbb{Z}\mathbb{A}_\infty$. Therefore $\Gamma_\cd$ has exactly $2r$ components of shape $\mathbb{Z}\mathbb{A}_\infty$. If $Q$ is not of type $\tilde{\mathbb{A}}$, then $\tilde{Q}$ is not of infinite Dynkin type. According to \cite[Theorem 6.6(1)]{BautistaLiuPaquette13}, $\mathcal{R}$ consists of infinitely many components of shape $\mathbb{Z}\mathbb{A}_\infty$. Therefore $\Gamma_\cd$ has infinitely many components of shape $\mathbb{Z}\mathbb{A}_\infty$.

(b) Assume that $Q$ is a graded oriented cycle. Then $\tilde{Q}$ is the linear quiver of type $\mathbb{A}^\infty_\infty$. Since all projective representations of $\tilde{Q}$ are infinite-dimensional, it follows from \cite[Lemma 4.5(1)]{BautistaLiuPaquette13} that the preprojective component of $\Gamma_{\rep^+(\tilde{Q})}$ consists of projective representations only and is of type $\tilde{Q}^{op}$ (=$\tilde{Q}$). Moreover, $\Gamma_{\rep^+(\tilde{Q})}$ has no preinjective components because such components contain finite-dimensional injective representations but all injective representations of $\tilde{Q}$ are infinite-dimensional. By \cite[Theorem 5.17(1)]{BautistaLiuPaquette13}, $\mathcal{R}$ consists of one component of type $\mathbb{Z}\mathbb{A}_\infty$. Therefore, $\Gamma_\cd$ consists of $r$ components of shape $\tilde{Q}^{op}$ and $r$ components of shape $\mathbb{Z}\mathbb{A}_\infty$.

(c) Assume that $Q$ has oriented cycles but is not a graded oriented cycle. Then $\tilde{Q}$ is not of infinite Dynkin type. It has both left infinite paths and right infinite paths. So $\cc_{\tilde{Q}}$ is a full subquiver of $\mathbb{Z}\tilde{Q}^{op}$ which has a right-most section and which contains finite Auslander--Reiten orbits only, by \cite[Theorems 4.6(2) and 4.7(2)]{BautistaLiuPaquette13}. By \cite[Theorem 6.6(4)]{BautistaLiuPaquette13},  all components of $\mathcal{R}$ are of shape $\mathbb{Z}\mathbb{A}_\infty$, $\mathbb{N}^-\mathbb{A}_\infty$, $\mathbb{N}\mathbb{A}_\infty$ or finite wings and there are infinitely many finite wings. This proves (c).
\end{proof}

\subsection{Special case: $Q$ is concentrated in degree $0$} 
Assume that $Q$ is concentrated in degree $0$ and is connected. Recall that in this case $\cd=\cd_{fd}(A)\simeq\cd^b(\mod A)$. 
 This category was studied in \cite{BautistaLiu06,BekkertDrozd09}. In particular, the set of indecomposable objects is described in terms of quiver representations (\cite[Theorem 3.11]{BautistaLiu06} and \cite[Theorem 4.1]{BekkertDrozd09}); the Auslander--Reiten theory of $\cd^b(\mod A)$  is studied in \cite[Section 4]{BautistaLiu06} and \cite[Section 5]{BautistaLiu17} and the possible shapes of the connected components of the Auslander--Reiten quiver of $\cd^b(\mod A)$ are described there.
In this case the quiver $P$ is the opposite quiver of Bekkert--Drozd's $Q^{\Box}$ in \cite[Section 2]{BekkertDrozd09} and is the quiver $Q^{\mathbb{Z}}$ in \cite[Section 7]{BautistaLiu14}; $\tilde{Q}$ is Bautista--Liu's minimal gradable covering of $Q$ in \cite[Theorem 1.3]{BautistaLiu06} and is the opposite quiver of Bekkert--Drozd's $Q^\diamond$ in \cite[Lemma 2.6]{BekkertDrozd09}.

Assume $r\geq 1$. Consider the composition
\[
\cf': \rep^-(\tilde{Q})\to \cd^b(\rep^-(\tilde{Q}))\to \cd^b(\rep^-(\tilde{Q}))/\Sigma^r(s^*)^{-1}\to\cd^b(\mod A),
\]
where the first functor is the canonical embedding, the second functor is the projection functor and the third functor is the triangle equivalence in Corollary~\ref{cor:D-as-orbit-quiver-description-3} composed with the triangle equivalence $\cd\simeq\cd^b(\mod A)$.  In \cite[Section 3]{BautistaLiu06}, a similar functor $\cf$ is explicitly constructed (it is denoted by $\cf_\pi$ in \cite[Section 4]{BautistaLiu17}).  We do not know if $\cf'$ coincides with $\cf$, but it does enjoy many properties of $\cf$, for example, as consequences of Corollary~\ref{cor:D-as-orbit-quiver-description-3} we have
\begin{itemize}
\item[--] $\cf'$ is fully faithful if $r\neq 1$ and is not fully faithful if $r=1$ (\confer \cite[Lemma 3.9]{BautistaLiu06}).
\item[--] $\cf'(M)[i]$, where $M$ runs over a complete set of pairwise non-isomorphic indecomposable objects of $\rep^-(\tilde{Q})$  and $0\leq i\leq r-1$, is a complete set of pairwise non-isomorphic indecomposable objects of $\cd^b(\mod A)$ (\confer \cite[Theorem 3.11(2)]{BautistaLiu06} and \cite[Theorem 4.11(2)]{BautistaLiu17}).
\end{itemize}


\def\cprime{$'$}
\providecommand{\bysame}{\leavevmode\hbox to3em{\hrulefill}\thinspace}
\providecommand{\MR}{\relax\ifhmode\unskip\space\fi MR }
\providecommand{\MRhref}[2]{%
  \href{http://www.ams.org/mathscinet-getitem?mr=#1}{#2}
}
\providecommand{\href}[2]{#2}

\end{document}